\documentclass[11pt]{amsart}
\usepackage{amssymb, latexsym, mathrsfs,   color }
\usepackage[all]{xy}
\usepackage[colorlinks=true, pdfstartview=FitV,linkcolor=blue,citecolor=blue,urlcolor=blue]{hyperref}

    \advance \topmargin by -\headheight
    \advance \topmargin by -\headsep

    \evensidemargin \oddsidemargin
\newtheorem {theorem}    {Theorem}[section]

\newtheorem {lemma}      [theorem]    {Lemma}
\newtheorem {corollary}  [theorem]    {Corollary}
\newtheorem {proposition}[theorem]    {Proposition}

\renewcommand{\rm}{\mathrm}

\theoremstyle{definition}

\numberwithin{equation}{section}

\begin{document}

\title{Remarks on the theta correspondence over finite fields}

\date{\today}

\author[Dongwen Liu]{Dongwen Liu}

\address{School of Mathematical Science, Zhejiang University, Hangzhou 310027, Zhejiang, P.R. China}

\email{maliu@zju.edu.cn}

\author[Zhicheng Wang]{Zhicheng Wang}

\address{School of Mathematical Science, Zhejiang University, Hangzhou 310027, Zhejiang, P.R. China}

\email{11735009@zju.edu.cn}


\begin{abstract}
S.-Y. Pan decomposed the uniform projection of the Weil representation of a finite symplectic-odd orthogonal dual pair, in terms of Deligne-Lusztig virtual characters, assuming that the order of the finite field is large enough. In this paper we use Pan's decomposition to study the theta correspondence for this kind of dual pairs, following the approach of Adams-Moy and Aubert-Michel-Rouquier. Our results give the theta correspondence between unipotent representations and certain quadratic unipotent representations.
\end{abstract}

\maketitle

\section{Introduction}

The theory of Weil representations and theta correspondence over finite fields has been studied extensively (see  e.g. \cite{AM, AKP, AMR, E, G, GH, H, P2, P3, P4, S1, S2}) and has many applications to the representation theory of $p$-adic groups, but there are still a few important open problems unsolved yet  in the finite field case. In particular, one has a complete understanding for the theta correspondence of unipotent representations for finite reductive dual pairs, with the only exception of symplectic-odd orthogonal dual pairs. The aim of this paper is to fill this gap and give some answers for  this missing case. In contrast to other dual pairs, it is well-known that the theta lifting of unipotent representations may not be unipotent in general for symplectic-odd orthogonal dual pairs. It turns out that new phenomena occurs  and the correspondence  involves certain quadratic unipotent representations, which are called $\theta$-representations in this paper.

To give some details, let us start from the Deligne-Lusztig characters \cite{DL}, which is a major tool for the study of representations of finite Lie groups. Let $\textbf{F}$ be a finite field of $q$ elements, $\textbf{G}$ be a connected reductive group defined over ${\bf F}$, and $F$ be the corresponding Frobenius endomorphism. Let $G = \textbf{G}^{F}$ be the group of rational points of $\textbf{G}$. For an $F$-stable maximal torus $\textbf{T}$ of ${\bf G}$ and a  character $\theta$ of $T = \textbf{T}^{F}$, let $R^{G}_{T,\theta}$ denote the virtual character of $G$ defined by P. Deligne and G. Lusztig in the seminal work \cite{DL}.

\par Let $\textbf{T}_{0}$ be a maximally split $F$-stable maximal torus of $\textbf{G}$, and $W_{\textbf{G}} = \textbf{N}(\textbf{T}_{0})/\textbf{T}_{0}$ be the Weyl group of $\textbf{G}$. For each $w \in W_{\textbf{G}}$, one can associate an $F$-stable maximal torus ${\bf T}_w$ of ${\bf G}$ and a specific character of $T_w={\bf T}_w^F$ of order 2, denoted by $\theta_{w}$ \cite{P2}. If $\textbf{G}= {\bf SO}_{2n+1}$ where $n$ is a positive integer, then
\[
\frac{1}{|W_{SO_{2n+1}}|}\sum_{w\in W_{SO_{2n+1}}}R^{SO_{2n+1}}_{T_{w},\theta_{w}}
\]
is the unique linear character of $SO_{2n+1}(q)$ of order 2, which will be denoted by $\chi_{SO_{2n+1}}$. For $n=0$, by convention $\chi_{SO_{2n+1}}$ is the trivial character. An irreducible representation $\pi$ of $G$ is called  a {\it $\theta$-representation} if it occurs in some $R^{G}_{T_{w},\theta_{w}}$. The $\theta$-representations can be interpreted as certain quadratic unipotent representations of finite reductive Lie groups (see e.g. \cite{M} for the notion of quadratic unipotent representations for $p$-adic groups). For $SO_{2n+1}$, $\pi \to \chi_{SO_{2n+1}}\pi$ gives a bijection between unipotent representations and $\theta$-representations.

\par Let $\omega_{Sp_{2n}}$ be the Weil representation of the finite symplectic group $Sp_{2n}(q)$ with $q$ odd, which depends on a choice of a nontrivial additive character $\psi$ of $\textbf{F}$ (cf. \cite{G}). We will denote by $\omega_{Sp_{2n}}^{\#}$ the uniform projection of $\omega_{Sp_{2n}}$, i.e. the projection onto the subspace of virtual characters spanned by all the Deligne-Lusztig characters. A pair $(G,G')$ of mutually centralized reductive subgroups of $Sp_{2n}$ is called a reductive dual pair. Then the Weil representation $\omega_{Sp_{2n}}$ restricts to a representation $\omega_{G,G'}$ of $G\times G'$, whose character will be denoted by the same notation.

\par In \cite{AM}, J. Adams and A. Moy show that for unitary or symplectic-even orthogonal dual pairs, if $\pi\otimes\pi'$ occurs in $\omega_{G,G'}$, then $\pi$ is unipotent if and only if $\pi'$ is unipotent, in which case they are in correspondence under the decomposition of the uniform projection $\omega_{G,G'}^\#$. They also study the first occurrence of cuspidal representations and give the explicit correspondence of cuspidal unipotent representations. 

As we mentioned above, it is known that for symplectic-odd orthogonal dual pairs, theta correspondence no longer  preserves unipotent representations in general. On the other hand in \cite{P2} S.-Y. Pan gives the decomposition of $\omega_{G,G'}^{\#}$ for such dual pairs. It turns out that based on Pan's work we can obtain the following analog of the result in \cite{AM}, which fills the gap in the symplectic-odd orthogonal case. See Theorem \ref{thm3.7} for details, and  also Theorem \ref{uni} for a general statement.

  \begin{theorem} \label{thm1.1}
Assume that $q$ is odd and large enough such that the main result in \cite{S2} holds. Let $\lambda_k$ be the unique cuspidal unipotent representation of $Sp_{2n}$, $n=k(k+1)$. Then its theta lifting to $O^\epsilon_{2n+1}$, $\epsilon=\pm$, is the unique cuspidal $\theta$-representation of $O^\epsilon_{2n+1}$ with trivial central character, and is the first occurrence in each Witt tower of odd orthogonal groups.
\end{theorem}

The proof of this theorem invokes some ingredients, among which we would like to mention the so-called conservation relation (see e.g. \cite{P1, SZ}). Such a relation provides a link between the first occurrence of the theta lifting of a representation towards two different Witt towers. For finite dual pairs it is known only for cuspidal representations, which is sufficient for our purpose.

Applying Lusztig's bijection \cite{L3} we are able to classify the cuspidal $\theta$-representations of symplectic groups (Theorem \ref{sptheta}). Consequently we can switch the roles of symplectic and odd orthogonal groups in Theorem \ref{thm1.1} and obtain an analogous result Theorem \ref{spthetalift}, the proof of which uses the conservation relations in an inductive manner.

The cuspidal representations are the building blocks of the so-called Harish-Chandra series. As a second step, A.-M. Aubert, J. Michel and R. Rouquier \cite[Th\'eor\`eme 3.7]{AMR} study the theta correspondence of Harish-Chandra series of unipotent cuspidal representations for unitary and symplectic-even orthogonal dual pairs. According to Lusztig's bijection and \cite[Theorem 3.10]{P1}, using their results one should be able to handle the theta correspondence of any irreducible representations. Based on Theorem \ref{thm1.1}, we are able to extend their results to the symplectic-odd orthogonal case using similar calculations as in \cite{AMR}. 

Our second main result Theorem \ref{amr2} establishes the theta correspondence between the Harish-Chandra series of cuspidal (quadratic) unipotent  representations of symplectic-odd orthogonal dual pairs.  We briefly describe its content as follows. Denote $T_l$  the split torus over ${\bf F}$ of rank $l$, and let $\theta_l$, $\theta_{k, l}$ and $\theta'_{k,l}$ be the quadratic characters of $T_l$ defined in \S\ref{sec4}.  Let $(G_n, G'_{n'})$ be a dual pair of symplectic and odd orthogonal groups over ${\bf F}$, and $\lambda$ be an irreducible cuspidal representation of $G_n$ such that its theta lifting $\lambda'$ to $G'_{n'}$ is irreducible cuspidal as well.  Let $\gamma$ be an irreducible constituent of the parabolic induction 
$R^{G_m}_{G_n\times T_{m-n}}(\lambda\times\theta)$, where $\theta$ is some quadratic character specified below, and $\Theta_{G'_{m'}}(\gamma)$ be the theta lifting of $\gamma$ to $G'_{m'}$.  Then Theorem \ref{amr2} states that $\Theta_{G'_{m'}}(\gamma)=0$ if $m'<n'$, and is spanned by irreducible constituents of $R^{G'_{m'}}_{G'_{n'}\times T_{m'-n'}}(\lambda'\otimes\theta')$ otherwise, where the pair $(\theta, \theta')$ of quadratic characters is as follows:

\begin{itemize}
\item If $G_n$ is symplectic and $G'_{n'}$ is odd orthogonal, then $(\theta, \theta')=(\theta_{m-n}, {\bf 1})$ or $({\bf 1}, \theta'_{m-n, m'-n'})$;

\item If $G_n$ is odd orthogonal and $G'_{n'}$ is symplectic, then $(\theta, \theta')=({\bf 1}, \theta_{m'-n'})$ or $(\theta_{m-n}, \theta_{m-n, m'-n'})$.
\end{itemize}
We expect that our results may have some applications towards the theta lifting of quadratic unipotent representations of $p$-adic groups.

From Mackey theory one knows that constituents of such Harish-Chandra series can be parametrized by irreducible representations of certain Weyl groups. Thus the theta correspondence reduces to a correspondence of Weyl group representations. In \cite{AMR} explicit formulas have been proved for unitary case and conjectured for symplectic-even orthogonal case, and latter has been proved recently in \cite{P3}. We expect that similar conjectures should also hold for symplectic-odd orthogonal case. We also remark that from the explicit correspondence one can observe that the conservation relations in general do not hold for noncuspidal representations of finite dual pairs. See \cite[Section 9]{P4} for more results on the (non-)conservation relations.

The paper is organized as follows. In \S\ref{sec2} we briefly recall the theory of Deligne-Lusztig characters as well as Pan's decomposition formula \cite{P2}. In \S\ref{sec3} we prove the theta correspondence between cuspidal unipotent representations and cuspidal $\theta$-representations, following \cite{AM} and using the conservation relations. Based on this result, in \S\ref{sec4} we establish the theta correspondence of certain Harish-Chandra series, using some calculations of Jacquet and induction functors as in \cite{AMR} and \cite{MVW}. We explain briefly the correspondence of Weyl group representations mentioned above at the end of \S\ref{sec4}.

Notations: In this paper the cardinality $q$ of ${\bf F}$ is assumed to be odd and large enough (as in Theorem \ref{pan}). The finite Lie group $G={\bf G}^F$ is also written as $G(q)$. Representations are realized on vector spaces over $\overline{\bf Q}_\ell$ where $\ell\neq p$ is a prime. We do not distinguish a (virtual) representation and its character. Thus $R(G)$ stands for the set of virtual representations of $G$ as well as the set of integral combinations of irreducible characters. For $\pi\in  R(G)$, let $\rm{Irr}(\pi)$ be the set of constituents of $\pi$. For $\rho\in \rm{Irr}(G)$, we simply write $\rho\in\pi$ if $\rho\in \rm{Irr}(\pi)$. For $\pi_1, \pi_2\in R(G)$, let
\[
(\pi_1,\pi_2)_G=\frac{1}{|G|}\sum_{g\in G}\pi_1(g)\overline{\pi_2(g)},
\]
where the conjugate is the restriction of the usual complex conjugate to $\overline{\bf Q}\subset\overline{\bf Q}_\ell$.

{\bf Acknowledgement.} The first named author would like to thank Jiajun Ma for many helpful and enlightening discussions on this topic. We thank the anonymous referee for many valuable comments.

\section{Deligne-Lusztig theory} \label{sec2}
Let $\textbf{F}$ be a finite field of $q$ elements, $\textbf{G}$ be a connect reductive group defined over ${\bf F}$ and $F$ be the corresponding Frobenius endomorphism. The group of rational points of $\textbf{G}$ is denoted by $G=\textbf{G}^{F}$.

\subsection{Conjugacy classes of maximal tori}
For an $F$-stable maximal torus $\textbf{T}$ of $\textbf{G}$, let $W_{\textbf{G}}(\textbf{T})= \textbf{N}_{\textbf{G}}(\textbf{T})/\textbf{T}$ be its Weyl group, where $\textbf{N}_{\textbf{G}}(\textbf{T})$ is the normalizer of $\textbf{T}$ in $\textbf{G}$, and let $T={\bf T}^F$ be the group of rational points of ${\bf T}$.

\par  Fix a maximally split $F$-stable maximal torus $\textbf{T}$$_{0}$ of $\textbf{G}$, and denote its Weyl group by $W_G$. A conjugate $^{g}\textbf{T}$$_{0}$ =$g\textbf{T}$$_{0}g^{-1}$ of ${\bf T}_0$ by an element $g\in \textbf{G}$ is $F$-stable if and only if $g^{-1}F(g)\in \textbf{N}_{\textbf{G}}(\textbf{T}_{0})$. By the Lang-Steinberg theorem, for each $w \in W_G$ we may choose an element $g \in \textbf{G}$ such that $g^{-1}F(g)$ is in $\textbf{N}_{\textbf{G}}(\textbf{T}_{0})$ whose image in $W_G$ is equal to $w$. The $F$-stable maximal torus $^{g}\textbf{T}$$_{0}$ will be denoted by $\textbf{T}$$_{w}$. Two elements $w, w' \in W_G$ are called $F$-conjugate if there exists $x\in W_G$ such that $w' = x^{-1}wF(x)$. The map $^{g}\textbf{T}$$_{0} \to g^{-1}F(g)$ gives a bijection between the $G$-conjugacy classes of $F$-stable maximal tori of ${\bf G}$ and the $F$-conjugacy classes of $W_G$.

\subsection{Deligne-Lusztig characters}
In their celebrated paper \cite{DL}, P. Deligne and G. Lusztig defined a virtual representation $R^{G}_{T,\theta}$ of $G$, associated to a character $\theta$ of $T$. The construction of Deligne-Lusztig characters makes use of the deep theory of $\ell$-adic cohomology. Here we only review some standard facts on these characters which will be used in this paper (cf. \cite[Chapter 7]{C}).
\par If $y = su$ is the Jordan decomposition of an element $y \in G$, then
\begin{equation}\label{dl}
R^{G}_{T,\theta}(y)=\frac{1}{|C^{0}(s)|}\sum_{g\in G,s^g \in T}\theta(s^{g})Q^{C^{0}(s)}_{^{g}T}(u)
\end{equation}
where $C^{0}(s) = C^{0}_{G}(s)$ is the connected component of the centralizer of $s$ in $G$, and $Q^{C^{0}(s)}_{^{g}T} = R^{ C^{0}(s)}_{^{g}T,1}(u)$ is the Green function of $C^{0}(s)$ associated to $^{g}T$. Note that $s^{g}=g^{-1}sg \in T$ if and only if $^{g}T=gTg^{-1} \subset C^{0}(s)$.

 It is known that
\begin{equation}\label{triv}
1=\frac{1}{|W_{G}|}\sum_{w\in W_{G}}R^{G}_{T_{w},1}=\sum_{(T)\subset G}\frac{1}{|W_G(T)|}R^{G}_{T,1}
\end{equation}
where $(T) \subset G$ means that the summation is taken over the $G$-conjugacy classes of $F$-stable maximal tori in $G$.

Two virtual representations $R^G_{T,\theta}$ and $R^G_{T',\theta'}$ are disjoint  if $(T,\theta)$ and $(T,\theta')$ are not geometrically conjugate. The elements in the subspace of class functions on $G$ spanned by all the $R_{T,\theta}^{G}$'s are called uniform functions of $G$. For a class function $f$ on $G$, denote by $f^\#$ its orthogonal projection to the subspace of uniform functions.

\subsection{Weyl group}
By \cite[Lemma 3.1]{S2}, there is a natural bijection between conjugacy classes of maximal tori in $Sp_{2n}(q)$ and the disjoint union of the $O^\epsilon_{2n}(q)$-conjugacy classes of maximal tori in $SO^\epsilon_{2n}(q)$ for $\epsilon=\pm$. For $w \in W_{Sp_{2n}}$, we define $\epsilon_{w}= \epsilon$ if the $F$-stable maximal torus $T_{w}$ of $Sp_{2n}(q)$ corresponds to an $F$-stable maximal torus in $SO^\epsilon_{2n}(q)$, and define
\[
W^{\epsilon}_{Sp_{2n}} = \{ w \in W_{Sp_{2n}} | \epsilon_{w} = \epsilon \}.
\]
We shall identify the Weyl groups of $Sp_{2n}$ and $SO_{2n+1}$, which will be denoted by $W_{n}$. Thus we also write $W_{n}^\epsilon = W_{Sp_{2n}}^\epsilon$. The Weyl group $W_{SO^{+}_{2n}}$ of $SO^+_{2n}$ can be regarded as a subgroup of $W_{n}$ of index 2.

\par Clearly, $\epsilon_{w} = +$ if $T_{w}$ $\simeq GL_{1}(q)$ in $Sp_{2}(q)$, and $\epsilon_{w} = -$ if $T_{w}$ $\simeq U_{1}(q)$ in $Sp_{2}(q)$. Since $Sp_{2n}(q^{t})$ can be regarded as a subgroup of $Sp_{2nt}(q)$ for $t \in \mathbb{N}$, an $F$-stable maximal torus of $Sp_{2n}(q^{t})$ is also an $F$-stable maximal torus of $Sp_{2nt}(q)$. Hence $W_{Sp_{2n}(q^{t})}$ can be embedded as a subset of $W_{Sp_{2nt}}(q)$. For $w \in W_{Sp_{2n}(q^{t})}$, the value of $\epsilon_{w}$ is the same no matter whether $w$ is regarded as an element of $W_{Sp_{2n}(q^{t})}$ or $W_{Sp_{2nt}(q)}$, and therefore $\epsilon_{w} = +$ (resp. $\epsilon_{w} = -$) if $T \cong GL_{1}(q^t)$ (resp. $T_{w}\cong U_{1}(q^{t}))$ in $Sp_{2t}(q)$.

 \subsection{The character $\theta_{w}$} \label{sec2.4}
If $T\cong GL_{1}(q^{t})$ or $U_{1}(q^{t})$, let $\theta_{T}$ denote the unique character of $T$ of order 2, i.e. $\theta_{T}(a)=\nu(a)^{\frac{q^{t}-1}{2}} $ (resp. $\theta_{T}(a)=\nu(a)^{\frac{q^{t} +1}{2}} $) for $a \in T$ if $T \cong GL_{1}(q^{t})$ (resp. $T \cong U_{1}(q^{t})$), where $\nu$ is an isomorphism of $T$ onto a subgroup of $\overline{\bf Q}_\ell^{\times}$. Define $\theta_{T}=\theta_{T_{1}}\otimes ...\otimes \theta_{T_{r}}$ if $T \cong T_{1} \times ... \times T_{r}$ is an $F$-stable maximal torus of a connected reductive group $G$, where $T_{i} \cong GL_{1}(q^{t^{i}})$ or $U_{1}(q^{t^{i}}) $ for some positive integer $t_{i}$, $i = 1, . . . , r$. It is clear that $^{g}\theta_{T} = \theta_{^{g}T}$ for $g \in G$. If $T = T_{w}$ for some $w \in W_{G}$, the character $\theta_{T}$ will be denoted by $\theta_{w}$.

  \par One can check that $\theta_{T} = \chi_{G}|_T$ if $G = GL_{n}$, $U_{n}$ or $SO_n$ ($n$ a positive integer), where $\chi_{G}$ is the unique linear character of $G$ of order 2. It is known as the spinor norm character in the case of orthogonal groups. For a representation or a character $\pi$ of $G$, we abbreviate $\chi_{G}\otimes \pi$ by $\chi_G\pi$.

 \par The following identity can be found in \cite{P2}.

   \begin{lemma}\label{lem2.1} If $G=GL_n$, $U_n$ or $SO_n$, then
  \[
\chi_{G}=\frac{1}{|W_{G}|}\sum_{w\in W_{G}}R^{G}_{T_{w},\theta_{w}}.
\]
\end{lemma}

  \begin{lemma} \label{lem2.2}
If $u \in G$ is an unipotent element, then $\chi_{G}(u)=1$.
\end{lemma}

\begin{proof} By (\ref{dl}), (\ref{triv}) and Lemma \ref{lem2.1},
\[
\begin{aligned}
\chi_{G}(u)&=\frac{1}{|W_{G}|}\sum_{w\in W_{G}}R^{G}_{T_{w},\theta_{w}}(u)
\\&=\frac{1}{|W_{G}|}\sum_{w\in W_{G}}\frac{1}{|G|}\sum_{g\in G}\theta_{w}(1)Q^{G}_{^{g}T}(u)
\\&=\frac{1}{|W_{G}|}\sum_{w\in W_{G}}\frac{1}{|G|}\sum_{g\in G}Q^{G}_{^{g}T}(u)
\\&=1(u)
\\&=1.
\end{aligned}
\]
\end{proof}

Recall that an irreducible representation $\pi$ is called a unipotent representation if it occurs in some $R^{G}_{T,1}$. We have the following

  \begin{lemma} \label{lem2.3}
If $G=GL_n$, $U_n$ or $SO_n$, then
  \par (i)  $\pi$ is a unipotent representation if and only if $\chi_{G}\pi$ occurs in some $R^{G}_{T_w,\theta_{w}}$,
\par (ii) $R^{G}_{T_{w},\theta_{w}}$ only consists of $\chi_{G}\pi$ with $\pi$ unipotent.
\end{lemma}

\begin{proof}
If $y = su$ is the Jordan decomposition of an element $y \in G$, then
\[
\begin{aligned}
\chi_{G}R^{G}_{T_{w},1}(y)&=\chi_{G}(y)\frac{1}{|C^{0}(s)|}\sum_{g\in G,s^{G} \in T_{w}}1(s^{g})Q^{C^{0}(s)}_{^{g}T_{w}}(u)
\\&=\frac{1}{|C^{0}(s)|}\sum_{g\in G,s^{G} \in T_{w}}\chi_{G}(s^{g})\chi_{G}(u)Q^{C^{0}(s)}_{^{g}T_{w}}(u)
\\&=\frac{1}{|C^{0}(s)|}\sum_{g\in G,s^{G} \in T_{w}}\theta_{w}(s^{g})Q^{C^{0}(s)}_{^{g}T_{w}}(u)
\\&=R^{G}_{T_{w},\theta_{w}}(y)
\end{aligned}
\]
\end{proof}

Similar to the notion of unipotent representations, a representation $\pi$ is called a $\theta$-representation if it occurs in some $R^{G}_{T_w,\theta_{w}}$. For $G=Sp_{2n}$, we do not have nontrivial characters of $G$, but we can still talk about $\theta$-representations.

\section{Weil representations and cuspidal unipotent representations} \label{sec3}

J. Adams and A. Moy consider in \cite{AM} (see also \cite{AMR, AKP, E}) the question of how the unipotent representations behave under theta correspondence. In particular they show that theta correspondence sends a cuspidal unipotent representation of $G$ to a cuspidal unipotent representation of $G'$ if $(G,G')=(Sp_{2n},O^\epsilon_{2n'})$ or $(U_{n},U_{n'})$. In the proof they use the decomposition of the Weil representation of $(Sp_{2n},O^\epsilon_{2n'})$ and $(U_{n},U_{n'})$ given in \cite{S2}. In \cite{P2} S.-Y. Pan gives a decomposition of the Weil representation of a finite symplectic-odd orthogonal dual pair, which we have recalled as Theorem \ref{pan}. Thus we can give the correspondence of cuspidal unipotent representations of such dual pairs based on Pan's result, together with another ingredient called conservation relation.

\subsection{Weil representation and theta lifting} \label{2.1}

Let $\omega_{Sp_{2n}}$ be the character of the Weil representation (cf. \cite{G}) of the symplectic group $Sp_{2n}$ which depends on a  nontrivial additive character $\psi$ of $F$, and let $\omega^{\#}_{Sp_{2n}}$ denote the uniform projection of $\omega_{Sp_{2n}}$, i.e. the projection onto the subspace of virtual characters spanned by all the Deligne-Lusztig characters. Let $(G, G^{\prime})$ be a reductive dual pair in $Sp_{2n}(q)$. Write
$\omega_{G,G'}$ for the restriction of $\omega_{Sp_{2n}}$ to $G\times G'$. Then it decomposes into a direct sum
\[
\omega_{G,G'}=\bigoplus_{\pi,\pi'} m_{\pi,\pi '}\pi\otimes\pi '
\]
where $\pi$ and $\pi '$ run over irreducible representations of $G$ and $G'$ respectively. We can rearrange this sum to get
\[
\omega_{G,G'}=\bigoplus_{\pi} \pi\otimes\Theta_{G'}(\pi )
\]
 where $\Theta_{G'}(\pi ) = \bigoplus_{\pi'} m_{\pi,\pi '}\pi '$ is a (not necessarily irreducible) representation of $G'$, called the (big) theta lifting of $\pi$ to $G'$. The theta lifting from $G$ to $G'$ will be written as $\pi \mapsto \Theta_{G'}(\pi)$, and we have $\pi'\in \Theta_{G'}(\pi)$ if and only if $\pi\otimes\pi'\in\omega_{G,G'}$. We should mention that it is an important open problem to find some ``good" constituents from $\Theta_{G'}(\pi)$ (see \cite{E, GH} for different approaches and results in this direction).   By abuse of notation we also write $\omega_{G,G'}$ for its character, and write $\omega^{\#}_{G,G'}$ for the uniform projection of $\omega_{G,G'}$.

It is often more convenient to work with the families of dual pairs associated to Witt towers instead of  a single dual pair. Thus let us introduce systematically some notations for Witt towers.
We will consider two Witt towers $\mathcal{T}$ and $\mathcal{T}'$ such that $(G_{n}, G_{n'}')$ form a reductive dual pair of type I for any $G_{n} \in \mathcal{T}$ and $G_{n'}' \in \mathcal{T}'$. There are the following types of Witt towers.

 $\bullet$ For unitary groups there are two Witt towers ${\bf U}^+=\left\{U_{2n}\right\}_{n\geq 0}$ and ${\bf U}^-=\left\{U_{2n+1}\right\}_{n\geq 0}$.

 $\bullet$ For symplectic groups there is only one Witt tower ${\bf Sp}=\left\{Sp_{2n}\right\}_{n\geq 0}$.

 $\bullet$ For even orthogonal groups there are two Witt towers ${\bf O}^+_\rm{even}=\left\{O^+_{2n}\right\}_{n\geq 0}$ and ${\bf O}^-_\rm{even}=\left\{O^-_{2n}\right\}_{n\geq 1}$.

$\bullet$ For odd orthogonal groups there are two Witt towers as well ${\bf O}^\epsilon_\rm{odd}=\left\{O^\epsilon_{2n+1}\right\}_{n\geq 0}$, $\epsilon=\pm$.

For the dual pair  $(Sp_{2n},O^{\epsilon}_{2n'})$ where $\epsilon=\pm$, we write $\omega^\epsilon_{n,n'}$ instead of $\omega_{G,G'}$. Since there are also two different nondegenerate symmetric bilinear forms on an odd dimensional vector space over $\bf{F}$ (though their orthogonal groups are physically equal), we have two different types of symplectic-odd orthogonal dual pairs. Similarly we denote them by $(Sp_{2n},O^{\epsilon}_{2n'+1})$ and the corresponding Weil representations by $\omega^\epsilon_{n,n'}$, where $\epsilon=\pm$. For unitary groups  write $G_n^\epsilon=U_{2n}$ for $\epsilon=+$ or $U_{2n+1}$ for $\epsilon=-$. Then we denote the Weil representation of $(G_n^\epsilon, G_{n'}^{\epsilon'})$ by $\omega_{n,n'}^{\epsilon,\epsilon'}$. We will specify the dual pair we are working with when necessary, so it should cause no confusion when we use the notation $\omega^\epsilon_{n,n'}$ or $\omega^{\epsilon,\epsilon'}_{n,n'}$, and occasionally we suppress the signs to ease notations.

 Now we recall the main result of \cite{P2}, which will be a crucial tool used in this paper.

  \begin{theorem} \label{pan}
Suppose that q is sufficiently large such that the main theorem in \cite{S2} holds. Then for a positive integer n and a nonnegative integer $n'$, we have
\[\begin{split}\omega^{\#}_{Sp_{2n},SO_{2n'+1}}\cdot(1\otimes \chi_{SO_{2n'+1}})=&\sum_{k=0}^{\min(n,n')}\frac{1}{|W_{k}|}\frac{1}{|W_{Sp_{2(n-k)}}|}\frac{1}{|W_{SO_{2(n'-k)+1}}|}\sum_{v \in W_{k}}
\\
&\sum_{\theta \in \rm{Irr}(T_{v})}  \sum_{w \in W_{Sp_{2(n-k)}}}\sum_{w' \in W_{SO_{2(n'-k)+1}}}\epsilon_{w}R^{Sp_{2n}}_{T_v\times T_{w},\theta\otimes \theta_{w}}\\& \otimes R^{SO_{2n'+1}}_{T_{v}\times T_{w'},\theta\otimes \theta_{w'}}\end{split}\]
where $\rm{Irr}(T_v)$ denotes the set of irreducible characters of the torus $T_v$.
\end{theorem}

We should point out that the above formula for the uniform projection is independent of the choice of $SO_{2n'+1}=SO^\epsilon_{2n'+1}$. In the rest of this paper the finite field ${\bf F}$ is assumed to be large enough as in Theorem \ref{pan}.

\subsection{Correspondence of unipotent and $\theta$-representations} \label{2.2}

For finite unitary and symplectic-even orthogonal dual pairs, J. Adams and A. Moy showed that if $\pi\otimes\pi'$ occurs in $\omega_{n,n'}$, then $\pi$ is unipotent if and only if $\pi'$ is unipotent. This does not hold for finite symplectic-odd orthogonal dual pairs, in which case one of the results we will prove says that $\pi$ is unipotent if and only if $\pi'$ is a $\theta$-representation when $n=n'$. In a similar manner as in \cite{AM}, our proof makes use of the decomposition of the Weil representation given by Theorem \ref{pan} in the symplectic-odd orthogonal case.

\par The following well-known result can be found in \cite{AM, MVW}, which says that the first occurrence of theta lifting of an irreducible cuspidal representation in a given Witt tower is also irreducible cuspidal, while the later occurrences are not cuspidal.

\begin{theorem}\label{FO}
Let $(G_n,G'_{n'})$ be a dual pair of finite classical groups. Assume that $\pi$ is an irreducible cuspidal representation of $G_n$, and that

(i) $\pi\mapsto \pi'$ in the theta correspondence for $G_n\times G'_{n'}$,

(ii) $\pi$ does not occur in the theta correspondence for $G_n\times G'_{k}$ for any $k<n'$.

{\flushleft{Then $\pi'$ is irreducible cuspidal, and $\Theta_{G'_k}(\pi)$ has no cuspidal constituents for $k>n'$.}}
\end{theorem}

Recall that by   Lemma \ref{lem2.3}, the map $\pi \mapsto \chi_{G}\pi$ ($G\neq Sp_{2n}$) gives a bijection between the unipotent representations and $\theta$-representations of $G$, and by Lemma \ref{lem2.2} it preserves the set of cuspidal representations.

We have the following basic result.

\begin{theorem}\label{uni}
Suppose that an irreducible representation $\pi \otimes \pi '$ of  $(Sp_{2n}, SO^\epsilon_{2n'+1}) $ occurs in $\omega^\epsilon_{n,n'}$. Then the following hold.

(i) If $\pi$ is a unipotent representation of $Sp_{2n}$, then $n'\geq n$ and
 $\pi '\in R^{SO^\epsilon_{2n'+1}}_{T_{v}\times T_{w'},\theta_v\otimes1}$ for some $v\in W_n$ and $w'\in W_{n'-n}$.

(ii) If $\pi'$ is a $\theta$-representation of $SO^\epsilon_{2n'+1}$, then
$n\geq n'$ and $\pi \in R^{Sp_{2n}}_{T_{v}\times T_{w},1\otimes\theta_w}$ for some $v\in W_{n'}$ and $w\in W_{n-n'}$.

(iii)  $\pi$ is a $\theta$-representation of $Sp_{2n}$ if and only if $\pi'$ is a unipotent representation of $SO_{2n'+1}^\epsilon$.

\end{theorem}

\begin{proof}
The proof is similar in spirit to that of \cite[Theorem 3.5]{AM}. First of all, for $G=SO_{2n'+1}$ from $\theta_T=\chi_G|_T$ it follows that all the pairs $(T_{w},\theta_{w})$, $w\in W_{n}$ are in the same geometric conjugacy class. More generally, two pairs of the form $(T_{v_i}\times T_{w_i},1\times\theta_{w_i})$, $v_i\in W_{n_i}$, $w_i\in W_{n'-n_i}$, $i=1,2$ are geometrically conjugate if and only if $n_1=n_2$.

 We need to recall some basic facts. Denote $G \cdot G'$ by $H$, and the character $R^{G}_{T}(\theta)\times R^{G'}_{T'}(\theta')$ by $R^{H}_{T\cdot T'} (\theta \cdot \theta')$. These characters form an
orthogonal basis for the space of uniform class functions of $H$. The set of irreducible characters $\mathcal{E}(H)$ of $H$ can be partitioned
by geometric conjugacy classes (see e.g. \cite{L2})
\[
\mathcal{E}(H)=\coprod_{s}\mathcal{E}(H,(s))
\]
where $s=s_{G^{*}}\cdot s_{G'^{*}} $ runs over the semisimple conjugacy classes of the dual group $H^{*} = G^{*}\cdot G'^{*}$. The irreducible constituents of a single $R^{H}_{T\cdot T'} (\theta\times \theta')$ are all associated to one geometric conjugacy class $(s)$.
Let $\pi_{H}$ be a representation of $H$. Then we have
\[
\pi_{H}=\bigoplus_{s}\pi_{H}(s)
\]
where $s$ runs over the semisimple conjugacy classes in $H^*$ and $\pi_{H}(s)\in \mathcal{E}(H,(s))$ is a subrepresentation of $H$. It is known that  if $\pi\in \mathcal{E}(H,(s))$ and $\pi'\in \mathcal{E}(H,(s'))$   with $(s)\neq (s')$, then $(\pi,\pi')_H=0$. We also have
\[
\pi_{H}^{\#}=\bigoplus_{s}\pi_{H}(s)^{\#}.
\]
If $\pi_{H}(s)\ne0$, then it is a positive combination of representations in $\mathcal{E}(H,(s))$. It follows from the fact that the
regular representation is uniform that $\pi_{H}(s)^{\#}$ is nonzero.

Now let $\pi_H=\pi\otimes\pi'$, which occurs in $\omega_{n,n'}^\epsilon$ by assumption. To prove (i), we note that if $\pi$ is unipotent then $\pi_H\in \mathcal{E}(H,(s))$ with $s_{G^*}=1$. By the above argument, $\pi_H^\#=\pi_H(s)^\#\neq 0$ hence it occurs in some $R^{Sp_{2n}}_{T,1}\otimes R^{SO^\epsilon_{2n'+1}}_{T',\theta}$, which has to appear in a summand of $\omega_{n,n'}^\#$. Then the conclusion follows from Theorem \ref{pan} and the first paragraph of the proof. The proof of (ii) and (iii) is similar which will be omitted.
\end{proof}

In particular by Theorem \ref{uni} (i) and (ii), if $n=n'$ then $\pi$ is a unipotent representation of $Sp_{2n}$ if and only if $\pi'$ is a $\theta$-representation of $SO_{2n+1}^\epsilon$. Applying Theorem \ref{FO} and Theorem \ref{uni}, we shall describe the theta correspondence between cuspidal unipotent representations of $Sp_{2n}$
and cuspidal $\theta$-representations of $SO_{2n'+1}$. We remark that the latter notion is not vacuum thanks to the following simple lemma.

\begin{lemma}\label{lem3.3} If $\pi$ is a cuspidal unipotent representation of a finite classical group $G$ which is not symplectic, then $\chi_G\pi$ is a cuspidal $\theta$-representation.
\end{lemma}

\begin{proof}
For any unipotent element $u$, by Lemma \ref{lem2.2} we have
\[
\chi_G\pi(u)=\chi_G(u)\pi(u)=\pi(u).
\]
Hence $\chi_G\pi$ is cuspidal. It is a $\theta$-representation by Lemma \ref{lem2.3}.
\end{proof}

\subsection{First occurrence for symplectic-odd orthogonal pairs.} For an irreducible representation $\pi$ of $G_n$, the smallest integer $n'(\pi)$ such that $\pi$ occurs in $\omega_{n,n'(\pi)}$ is called the {\it first occurrence index} of $\pi$ in the Witt tower $\left\{G'_{n'}\right\}_{n'\geq 0}$. By \cite[Chap.3, lemme IV.2]{MVW}, there exists $n'$ such that $\Theta_{G'_{n'}}(\pi)\neq 0$, hence $n'(\pi)$ is well-defined. In this section we compute the first occurrence of cuspidal unipotent and $\theta$-representations for symplectic-odd orthogonal dual pairs. We begin by reviewing Lusztig's results \cite{L1} on the cuspidal unipotent representations of symplectic and orthogonal groups.

\begin{theorem} \label{fun} The following groups

(i) $U_n$, $n=k(k+1)/2$,

(ii) $Sp_{2n}$, $n=k(k+1)$,

(iii) $SO_{2n+1}$, $n=k(k+1)$,

(iv) $SO^{\epsilon}_{2n}$, $n=k^2$, $\epsilon=\rm{sgn}(-1)^k$
\\are the only groups in their respective Lie families which possess a cuspidal unipotent representation. In each case, the specified group $G$ has a unique cuspidal unipotent representation.
\end{theorem}

By Lemma \ref{lem3.3}, we have the following immediate consequence for cuspidal $\theta$-representations.

\begin{corollary} \label{ctheta}
The following groups

(i) $U_n$, $n=k(k+1)/2$,

(ii) $SO_{2n+1}$, $n=k(k+1)$,

(iii) $SO^{\epsilon}_{2n}$, $n=k^2$, $\epsilon=\rm{sgn}(-1)^k$
\\are the only groups in their respective Lie families which possess a cuspidal $\theta$-representation. In each case, the specified group $G$ has a unique cuspidal $\theta$-representation.
\end{corollary}

 We define a $\theta$-representation of $O_{2n+1}$, which is nonconnected, to be a constituent of $\rm{Ind}^{O_{2n+1}}_{SO_{2n+1}}\pi$ for a $\theta$-representation $\pi$ of $SO_{2n+1}$. Let $\lambda'_k$ stand for the unique cuspidal unipotent representation of $SO_{2n+1}$, $n=k(k+1)$. Then
\[
\lambda'_{k,\chi}:=\chi_{SO_{2n+1}}\lambda'_k
\]
is the unique cuspidal $\theta$-representation of $SO_{2n+1}$, and we have
  \[
  \rm{Ind}^{O_{2n+1}}_{SO_{2n+1}}\lambda'_{k,\chi}=\lambda'^+_{k,\chi}\oplus \lambda'^{-}_{k,\chi}
  \]
  where $\lambda'^+_{k,\chi}:= \lambda'_{k,\chi}\otimes 1$ and $\lambda'^-_{k,\chi}:=\lambda'_{k,\chi}\otimes\rm{sgn}$ are cuspidal  $\theta$-representations of $O_{2n+1}\cong SO_{2n+1}\times\{\pm I\}$.
Now we can give the first main result of this section, which is stated as Theorem \ref{thm1.1} in the Introduction.

  \begin{theorem}\label{thm3.7}
Let $\lambda_k$ be the unique cuspidal unipotent representation of $Sp_{2n}$, $n=k(k+1)$. Then the first occurrence index $n'^{\epsilon}(\lambda_k)$  in the Witt tower $\left\{O^\epsilon_{2n'+1}\right\}_{n'\geq 0}$ is equal to $n$, and the theta lifting $\Theta_{O^\epsilon_{2n+1}}(\lambda_k)$  equals $\lambda'^+_{k,\chi}$.
\end{theorem}

The proof of this theorem invokes three ingredients. The first one is the following information which can be read off of \cite{H}.

\begin{proposition} \label{howe}
Every irreducible representation of $Sp_{2n}$ is contained in $\omega^+_{n,n}\oplus\omega^-_{n,n}$, where $\omega^\epsilon_{n,n'}=\omega_{Sp_{2n},O^\epsilon_{2n'+1}}.$
\end{proposition}

The second ingredient is the following so-called conservation relation for cuspidal representations given in \cite{P1}.

\begin{proposition}\label{cons}
(i) Let $\pi$ be an irreducible cuspidal representation of $Sp_{2n}$ and $n'^\epsilon(\pi)$ be the first occurrence index of $\pi$ in the Witt tower
${\bf O}^\epsilon_\rm{odd}=\left\{O^\epsilon_{2n'+1}\right\}_{n'\geq 0}$. Then
\[
n'^+(\pi)+n'^-(\pi)=2n.
\]
(ii) Let $\pi'$ be an irreducible cuspidal representation of $O^\epsilon_{2n'+1}$ and $n(\pi')$ be the first occurrence index of $\pi'$ in the Witt tower
${\bf Sp}=\{Sp_{2n}\}_{n\geq 0}$. Then
\[
n(\pi')+n(\pi'\otimes\rm{sgn})=2n'+1.
\]
\end{proposition}

A few remarks are in order. In the literature the conservation relations are usually formulated  in terms of the dimensions of vector spaces, but in this paper we will use the above form for convenience. Here we only recalled the statement for symplectic-odd orthogonal dual pairs and we refer the readers to \cite{P1} for other cases. In the world of $p$-adic local fields, the conservation relations have been established in full generality in \cite{SZ}.  We should also mention that the standard formulation of conservation relations does not seem to hold in general for noncuspidal representations of finite dual pairs, and we will give some explanations at the end of the paper.

The last ingredient we need is that the action of $-I$ under a unipotent representation is trivial. This fact might be known but we could not find a reference, so we include a detailed proof for completeness. It can be used to determine whether the theta lifting of a unipotent representation takes place in the even or odd part of the Weil representation, but we will not need this result.

\begin{proposition}\label{pi-1}
(i) If $\pi$ is a unipotent representation of $GL_{n}$, $U_n$, $Sp_{2n}$ or $SO^\epsilon_{2n}$, then $\pi(-I)$ is trivial.

(ii)  If $\pi$ is a unipotent representation of $Sp_{2n}$ and $\pi\otimes \pi'$ occurs in the Weil representation $\omega^\epsilon_{n,n'}$ of
$(Sp_{2n},O^\epsilon_{2n'+1})$, then $\pi'(-I)$ is trivial.
\end{proposition}

\begin{proof} (ii) clearly follows from (i) because $\pi(-I)=\pi'(-I)$ through the Weil representation $\omega^\epsilon_{n,n'}$. We shall prove (i) for all the classical groups (except odd orthogonal groups) simultaneously by induction on $n$. We first observe that by (\ref{dl}),
\begin{equation}\label{dl-1}
R^G_{T,1}(-I)=R^G_{T,1}(I).
\end{equation}
If $n=1$, then $G$ has only two unipotent representations $1$ and $\pi$. Hence the conclusion follows from (\ref{dl-1}).
Assume that $n>1$ and $\pi\in R^G_{T,1}$ for some $T$. If $T$ is not anisotropic, then it is contained in a Levi subgroup $L$ of a proper parabolic subgroup $P$ of $G$, i.e. $P={\bf P}^F$ for some $F$-stable proper parabolic subgroup ${\bf P}$ of ${\bf G}$. By Frobenius reciprocity,
 \[
 (\pi,R^G_{T,1})_G=(\pi,\rm{Ind}^G_P  R^{L}_{T,1})_G=(J^G_P(\pi), R^{L}_{T,1})_L\ne 0,
 \]
where $J^G_P$ is the Jacquet functor.
We may write $L=GL_{n_{1}}\times GL_{n_{2}}\times ...\times  GL_{n_{r}}\times G_{n_0}$, where $n_i<n$, $i=0,\ldots,r$ and $G_{n_0}$ is a classical group of the same type as $G$. There exists $\pi_0\in J^G_P(\pi)$ such that $(\pi_0, R^L_{T,1})_L\neq 0$. Then by induction we have $\pi(-I)=\pi_0(-I)$ is trivial. By \cite[Theorem 6.25]{S3}, if $\pi\not\in R^{G}_{T,1}$ for any $T$ which is not anisotropic, then $\pi$ is a cuspidal representation. If $G=GL_n$ then we are done because it does not have cuspidal unipotent representations when $n>1$.
If $G\neq GL_n$, then it remains to consider the case that $\pi$ is the unique cuspidal unipotent representation of $G$ (if it exists).
 The uniqueness of $\pi$ enables us to write
 \[
 (\pi,R^G_{T,1})\pi=R^G_{T,1}-\sum_{\pi_{i}\in R^{G}_{T,1}, \pi_i\neq\pi}(\pi_i, R^{G}_{T,1})_G\pi_{i},
 \]
where the $\pi_{i}$'s are all unipotent representations but not cuspidal. Then the conclusion follows from (\ref{dl-1}) and that $\pi_i(-I)$ is trivial for each $i$.
\end{proof}


\begin{proof} (of Theorem \ref{thm3.7}) By Proposition \ref{howe}, the cuspidal unipotent representation $\lambda_k$ of $Sp_{2n}$ is contained in $\omega^+_{n,n}\oplus\omega^-_{n,n}$. By Theorem \ref{uni},  $\lambda_k$ is not contained in any $\omega^\epsilon_{n,n'}$ when $n'<n$, and for some $\epsilon_0\in\{+,-\}$ the theta lifting $\Theta_{O^{\epsilon_0}_{2n+1}}(\lambda_k)$ to $O^{\epsilon_0}_{2n+1}$ is the first occurrence in the Witt tower ${\bf O}^{\epsilon_0}_\rm{odd}$, hence is a cuspidal $\theta$-representation. Then it follows from Corollary \ref{ctheta} and Proposition \ref{pi-1} that $\Theta_{O^{\epsilon_0}_{2n+1}}(\lambda_k)=\lambda'^+_{k,\chi}$. The conservation relation Proposition \ref{cons} (i) implies that $n'^{-\epsilon_0}(\lambda_k)=2n-n'^{\epsilon_0}(\lambda_k)=n$, and one can similarly deduce that  $\Theta_{O^{-\epsilon_0}_{2n+1}}(\lambda_k)=\lambda'^+_{k,\chi}$ as well.
\end{proof}

We shall also study the first occurrence of the theta lifting of cuspidal $\theta$-representations of $Sp_{2n}$ in ${\bf O}^\epsilon_\rm{odd}$, which  by Theorem \ref{uni}
will be cuspidal unipotent representations.
By definition, a unipotent representation of $O_{2n+1}$ is a constituent of $\rm{Ind}^{O_{2n+1}}_{SO_{2n+1}}(\pi)$ for a unipotent representation $\pi$ of $SO_{2n+1}$. Let $\lambda'_{k}$ be the unique cuspidal unipotent representation of $SO_{2n+1}$, $n =k(k + 1)$.  Then we have
\begin{equation}\label{oddcu}
\rm{Ind}^{O_{2n+1}}_{SO_{2n+1}}(\lambda'_{k})=\lambda^{'+}_{k}\oplus\lambda_{k}^{'-}
\end{equation}
where $\lambda_k'^+:=\lambda_k'\otimes 1$ and $\lambda_k'^-:=\lambda_k'\otimes\rm{sgn}$ are cuspidal unipotent representations of $O_{2n+1}\cong SO_{2n+1}\times\{\pm I\}$.

Applying a result of Lusztig, we first obtain the following classification of cuspidal $\theta$-representations of finite symplectic groups.

\begin{theorem}\label{sptheta} The groups $Sp_{2n}$, $n=k^2$ are the only symplectic groups which possess cuspidal $\theta$-representations. Each $Sp_{2k^2}$, $k\neq 0$, has two cuspidal $\theta$-representations $\lambda_{k,\alpha}$ and $\lambda_{k,\beta}$, which satisfy $\lambda_{k,i}(-I)=(-1)^{k}$, $i=\alpha,\beta$.
\end{theorem}

\begin{proof} We first recall some general results. Assume that ${\bf G}$ is a simple simply connected linear algebraic group over ${\bf F}$, and ${\bf T}$ is an $F$-stable maximal torus of ${\bf G}$. Let ${\bf G}^{*}$ be the dual group of ${\bf G}$, and ${\bf T}^*$ be the $F$-stable maximal torus of ${\bf G}^*$ which is the dual of ${\bf T}$. Write $G^*={\bf G}^{*F}$, $T^*={\bf T}^{*F}$. Then a character $\theta$ of $T$ corresponds to some element $s \in T^{*}$, which is well-defined up to conjugacy by $W_{{\bf G}^*}({\bf T}^{*})^{F}$. In this case we write $R_{T^{*},s}^{G}$ instead of $R^G_{T,\theta}$.
Let $\mathcal{E}(G,(s))$ be the set of irreducible representations which occur in $R_{T^{*},s}^{G}$ for some $T^*$ such that $s\in T^*$. By \cite{L3} there is a bijection
\[
\mathcal{L}:\mathcal{E}(G,(s)) \to \mathcal{E}(C_{G^{*}}(s),(1)),
\]
where $\mathcal{E}(C_{G^*}(s),(1))$ is the set of irreducible representations of $C_{G^*}(s)$ whose restrictions to $C_{G^*}(s)^\circ$ are sums of unipotent representations of 
$C_{G^*}(s)^\circ$. Note that $C_{G^*}(s)$ may not be connected in general. 
Moreover if the identity components of the centers of ${\bf G}$ and $C_{{\bf G}^*}(s)$ have the same ${\bf F}$-rank, then $\pi\in \mathcal{E}(G,(s))$ is cuspidal if and only if  $\mathcal{L}(\pi)\in \mathcal{E}(C_{G^*}(s),(1))$ is cuspidal (see e.g. \cite[Chap. 9]{L1} and \cite[Lemma 2.7]{Ma}).

For $G=Sp_{2n}$ we have $G^*=SO_{2n+1}$. The pair $(T_w, \theta_w)$ corresponds to $s\in T_w$, where $s=(-I,1)$ with $I$ being the identity in $SO^{\epsilon_w}_{2n}\subset SO_{2n+1}$. Then $C_{G^*}(s)\cong O^{\epsilon_w}_{2n}\subset SO_{2n+1}$, hence by the above results we have a bijection between cuspidal $\theta$-representations of $Sp_{2n}$ and the union of cuspidal unipotent representations of $O^{\epsilon}_{2n}$ over $\epsilon=\pm$.
By Theorem \ref{fun}, each $O_{2k^2}^\epsilon$ with $\epsilon:=\rm{sgn}(-1)^k$ has two cuspidal unipotent representations, which are the two constituents induced from the unique cuspidal unipotent representation of $SO^\epsilon_{2k^2}$, and these are the only even orthogonal groups which possess cuspidal unipotent representations. Hence each $Sp_{2n}$, $n=k^2$ has two cuspidal $\theta$-representations $\lambda_{k,1}$ and $\lambda_{k,2}$, and these are the only possibilities.

From \cite{H} we also know that every irreducible representation of $Sp_{2n}$ is contained in $\omega_{Sp_{2n},O^+_{2n}}\oplus\omega_{Sp_{2n},O^-_{2n}}$. Each $\lambda_{k,i}$, $i=\alpha,\beta$ occurs in some $R^G_{T_w,\theta_w}$ such that $\epsilon_w=\epsilon=\rm{sgn}(-1)^k$.  Similar to the proof of Theorem \ref{uni} and Theorem \ref{thm3.7}, applying the main result of \cite{S2} as well as conservation relations one can show that $\lambda_{k,i}$ occurs in $\omega_{Sp_{2n}, O^\epsilon_{2n}}$ and $\Theta_{O^{\epsilon}_{2n}}(\lambda_{k_i})$, $i=\alpha,\beta$ are the two cuspidal $\theta$-representations of $O^\epsilon_{2n}$. Since $\chi_{SO^\pm_{2n}}(-I)=\pm1$, we obtain that
$\lambda_{k,i}(-I)=\chi_{SO^\epsilon}(-I)=(-1)^k$.
\end{proof}

Finally we can give the second main result of this section.

\begin{theorem}\label{spthetalift}
Let $\lambda_{k,i}$, $i=\alpha,\beta$ be the cuspidal $\theta$-representations of $Sp_{2k^2}$, and let $n'^{\epsilon}(\lambda_{k,i})$ be the first occurrence index of $\lambda_{k,i}$ in the Witt tower ${\bf O}^\epsilon_\rm{odd}$. Then for some $\epsilon_0 \in \{+,-\}$ one has $n'^{\epsilon_0}(\lambda_{k,\alpha})=n'^{-\epsilon_0}(\lambda_{k,\beta})=k(k-1)$ and $n'^{-\epsilon_0}(\lambda_{k,\alpha})=n'^{\epsilon_0}(\lambda_{k,\beta})=k(k+1)$. Write $\epsilon(k)=\rm{sgn}(-1)^k$. Then the theta liftings are given by
\begin{align*}
& \Theta_{O^{\epsilon_0}_{2k(k-1)+1}}(\lambda_{k,\alpha})=\lambda_{k-1}'^{\epsilon(k)},\quad \Theta_{O^{-\epsilon_0}_{2k(k+1)+1}}(\lambda_{k,\alpha})=\lambda_{k}'^{\epsilon(k)},\\
& \Theta_{O^{\epsilon_0}_{2k(k+1)+1}}(\lambda_{k,\beta})=\lambda_{k}'^{\epsilon(k)},\quad \Theta_{O^{-\epsilon_0}_{2k(k-1)+1}}(\lambda_{k,\beta})=\lambda_{k-1}'^{\epsilon(k)},
\end{align*}
where $\lambda'^\epsilon_k$ is defined by (\ref{oddcu}).
\end{theorem}

We remark that the $\epsilon_0$ above depends on the choice of the nontrivial character $\psi$ of ${\bf F}$ which we used to define the Weil representations. If $\psi$ is replaced by $\psi_t(x):=\psi(tx)$ for some $t\in {\bf F}^\times\setminus{\bf F}^{\times 2}$, then $\epsilon_0$ should be replaced by $-\epsilon_0$.

\begin{proof} We use an inductive argument which is parallel to the proof in \cite{AM}. By convention we let $\lambda_0$ be the trivial representation of $Sp_0$.

For $Sp_{2}$, $\lambda_{1,\alpha}$ is contained in $\omega^+_{1,1}\oplus\omega^-_{1,1}$. Since by Theorem \ref{uni} (iii) the first occurrence of the theta lifting of $\lambda_{1,\alpha}$ is  unipotent cuspidal, we see that $\lambda_{1,\alpha}$ occurs in $\omega^{\epsilon_0}_{1,0}$ for some $\epsilon_0$. By conservation relation, $\lambda_{1,\alpha}$ occurs in $\omega^{-\epsilon_{0}}_{1,2}$ as the first occurrence in ${\bf O}^{-\epsilon_0}_\rm{odd}$. By Theorem \ref{sptheta}, we have the theta liftings $\Theta_{O^{\epsilon_{0}}_{1}}(\lambda_{1,\alpha})=\lambda'^-_0$ and $\Theta_{O^{-\epsilon_{0}}_5}(\lambda_{1,\alpha})=\lambda'^-_2$. Similarly $\lambda_{1,\beta}$ is contained in $\omega^+_{1,0}\oplus\omega^-_{1,0}$.  It cannot happen that  $\Theta_{O^{\epsilon_{0}}_{1}}(\lambda_{1,\alpha})=\Theta_{O^{\epsilon_{0}}_{1}}(\lambda_{1,\beta})=\lambda'^{-}_{0}$, hence $\lambda_{1,\beta}$ must be contained in $\omega^{-\epsilon_0}_{1,0}$. By conservation again, $\lambda_{1,\beta}$ occurs in $\omega^{\epsilon_{0}}_{1,2}$ as the first occurrence in ${\bf O}^{\epsilon_0}_\rm{odd}$. The theta liftings of $\lambda_{1,\beta}$ to $O^{-\epsilon_0}_1$ and $O^{\epsilon_0}_5$ are therefore $\lambda'^{-}_{0}$ and $\lambda'^{-}_{2}$ respectively.
We may exhaust the theta correspondence inductively through the following diagram.
 \[\begin{matrix}
   Sp_{0} &\quad& \lambda_{0} &\quad& \longrightarrow &\quad& \lambda'^{+}_{0}= 1 &\quad& \omega^{\epsilon_0}_{0,0}\\
   Sp_{0} &\quad& \lambda_{0} &\quad& \longrightarrow &\quad& \lambda'^{+}_{0}=1 &\quad& \omega^{-\epsilon_0}_{0,0}\\
   Sp_{2} &\quad& \lambda_{1,\alpha} &\quad& \longrightarrow &\quad& \lambda'^{-}_{0}=\rm{sgn} &\quad& \omega^{\epsilon_{0}}_{1,0}\\
   Sp_{2} &\quad& \lambda_{1,\beta} &\quad& \longrightarrow &\quad& \lambda'^{-}_{0}=\rm{sgn} &\quad& \omega^{-\epsilon_{0}}_{1,0}\\
   Sp_{2} &\quad& \lambda_{1,\alpha} &\quad& \longrightarrow &\quad& \lambda'^{-}_{1} &\quad& \omega^{-\epsilon_{0}}_{1,2}\\
   Sp_{2} &\quad& \lambda_{1,\beta} &\quad& \longrightarrow &\quad& \lambda'^{-}_{1} &\quad& \omega^{\epsilon_{0}}_{1,2}\\
   \quad &\quad& \quad &\quad& \vdots &\quad& \quad &\quad & \quad\\
   Sp_{2k^2} &\quad& \lambda_{k,\alpha} &\quad& \longrightarrow &\quad& \lambda'^{\epsilon(k)}_{k-1} &\quad& \omega^{\epsilon_{0}}_{k^2,k(k-1)}\\
   Sp_{2k^2} &\quad& \lambda_{k,\beta} &\quad& \longrightarrow &\quad& \lambda'^{\epsilon(k)}_{k-1} &\quad& \omega^{-\epsilon_{0}}_{k^2,k(k-1)}\\
   Sp_{2k^2} &\quad& \lambda_{k,\alpha} &\quad& \longrightarrow &\quad& \lambda'^{\epsilon(k)}_{k} &\quad& \omega^{-\epsilon_{0}}_{k^2,k(k+1)}\\
   Sp_{2k^2} &\quad& \lambda_{k,\beta} &\quad& \longrightarrow &\quad& \lambda'^{\epsilon(k)}_{k} &\quad& \omega^{\epsilon_{0}}_{k^2,k(k+1)}\\
  \end{matrix}
\]

In general for $Sp_{2k^{2}}$,  $\lambda_{k,i}$, $i=\alpha,\beta$ are contained in $\omega^+_{k^{2},k^{2}}\oplus\omega^-_{k^{2},k^{2}}$ by Proposition \ref{howe}. Since the first occurrences of their theta liftings are cuspidal unipotent,  they are contained in $\omega^+_{k^{2},k(k-1)}\oplus\omega^-_{k^{2},k(k-1)}$. By induction we have
\[
\Theta_{O^{\epsilon_0}_{2k(k-1)+1}}(\lambda_{k-1,\beta})=\lambda'^{\epsilon(k-1)}_{k-1}.
\]
Then by the conservation Proposition \ref{cons} (ii) we may arrange the indices such that
\[
\Theta_{O^{\epsilon_0}_{2k(k-1)+1}}(\lambda_{k,\alpha})=\lambda'^{\epsilon(k-1)}_{k-1}\otimes \rm{sgn}=\lambda'^{\epsilon(k)}_{k-1}.
\]
By conservation again, $\lambda_{k,\alpha}$ occurs in $\omega^{-\epsilon_{0}}_{k^2,k(k+1)}$. Similarly, $\lambda_{k,\beta}$ occurs in $\omega^{-\epsilon_{0}}_{k^2,k(k-1)}$ and $\omega^{\epsilon_{0}}_{k^2,k(k+1)}$. The remaining assertions about the theta liftings are clear.
\end{proof}

\section{Theta correspondence and Harish-Chandra series} \label{sec4}

In this section we study the theta correspondence of representations in certain Harish-Chandra series in a way similar to \cite{AMR}.

Let $T_l$ be the split torus over ${\bf F}$ of rank $l$, and $\theta_l:=\theta_{T_l}$ the order 2 character of $T_l$ defined in \S\ref{sec2.4}. In particular $(T_1,\theta_1)=(GL_1,\chi_{GL_1})$. For later use, for $k, l\geq 0$ we introduce a pair of  quadratic characters $\theta_{k,l}$ and $\theta_{k,l}'$ of $T_l$ by
\begin{equation}\label{tkl}
\theta_{k,l}=\left\{\begin{array}{ll} 1_k\otimes \theta_{l-k}, & \textrm{ if }k\leq l,\\
1, & \textrm{ otherwise},\end{array}\right.
\end{equation}
and
\begin{equation}\label{tkl'}
\theta'_{k,l}=\theta_{k,l}\theta_l=\left\{\begin{array}{ll} \theta_k\otimes 1_{l-k}, & \textrm{ if }k\leq l,\\
\theta_l, & \textrm{ otherwise},\end{array}\right.
\end{equation}
where $1_k$ and $1_{l-k}$ are the trivial characters of $T_k$ and $T_{l-k}$ respectively.

For two groups $G_n\subset G_m$ in the same Witt tower, we can embed $G_{n}$ into the Levi subgroup $G_{n}\times T_{m-n}$ of a parabolic subgroup of $G_{m}$. Let $\lambda$ be a cuspidal representation of $G_{n}$, and denote by $R(G_{m})_{\lambda}$ the subset of $R(G_{m})$ whose elements are spanned by
\[
\rm{Irr}(G_{m})_{\lambda}:=\rm{Irr}\left(R^{G_{m}}_{G_{n} \times  T_{m-n}}(\lambda\otimes 1)\right).
\]
Similarly introduce the subset  $R(G_{m})_{\lambda ,\theta}$  of $R(G_{m})$ whose elements are spanned by
\[
\rm{Irr}(G_{m})_{\lambda,\theta}:=\rm{Irr}\left(R^{G_{m}}_{G_{n} \times T_{m-n}}(\lambda\otimes \theta_{{m-n}})\right).
\]
Note that in this context the functor $R$ above stands for the usual parabolic induction, which coincides with the generalized Deligne-Lusztig induction.

 The standard Levi subgroups $L$ of $G_m$ are of the form $L = GL_{n_{1}}(q)  \times \cdot\cdot\cdot \times GL_{n_{r}}(q) \times G_l(q)$ such that $m = n_{1}+\cdot\cdot\cdot+n_{r}+l$. A representation $\rho$ of $L$ is a cuspidal, unipotent or $\theta$-representation if it is of the form $\rho = \rho_{1} \otimes\cdot\cdot\cdot\otimes\rho_{r}\otimes\sigma$, where $\rho_i$ and $\sigma$ are cuspidal, unipotent or $\theta$-representations of $GL_{n_{i}}(q)$ and  $G_l(q)$ respectively.

For a finite (reductive) group $G$, let $(R^G,\overline{\bf Q}_\ell[G])$ be the regular representation of  $G\times G$ on the space of $\overline{\bf Q}_\ell$-valued functions $f$ on $G$ defined by
\[
(R^G(g_{1},g_{2})f)(g)=f(g_{1}^{-1}gg_{2}),  \qquad \forall g_{1},g,g_{2}\in G.
\]
Then $R^G$ has the well-known Peter-Weyl decomposition
\[
R^G=\sum_{\pi\in\rm{Irr}(G)} \pi\otimes\pi^\vee
\]
where $\pi^\vee$ is the contragredient of $\pi$. For a quadratic character $\chi$ of $G$, let $\chi  R^G$ be the representation of $G\times G$ defined by
\[
\chi R^G=\sum_{\pi \in\rm{Irr}(G)} \chi\pi\otimes \pi^\vee=\sum_{\pi\in\rm{Irr}(G)}\pi\otimes\chi\pi^\vee.
\]

\par The follow result given in \cite[Th\'eor\`eme 3.7]{AMR} shows that the theta correspondance is compatible with the Harish-Chandra decomposition and that it preserves cuspidal unipotent representations for the dual pairs $(U_{n}(q), U_{n'}(q))$ and $(Sp_{2n}(q), O^{\epsilon}_{ 2n'}(q))$.

\begin{theorem} \label{amr}
Assume that $(\mathcal{T},\mathcal{T}')=({\bf U}^\epsilon, {\bf U}^{\epsilon'})$, $({\bf Sp}, {\bf O}^\epsilon_{\rm{even}})$ or $({\bf O}^\epsilon_{\rm{even}}, {\bf Sp})$, $\epsilon,\epsilon'=\pm$. Let $\lambda$ be an irreducible cuspidal representation of $G_{n}\in \mathcal{T}$ and let $n'=n'(\lambda)$ be its first occurrence index in $\mathcal{T}'$, so that $\lambda':=\Theta_{G'_{n'}}(\lambda)$ is an irreducible cuspidal representation of $G'_{n'}\in\mathcal{T}'$. Then for $\gamma\in \rm{Irr}(G_m)_\lambda$, $\Theta_{G'_{m'}}(\lambda)=0$ whenever $m'<n'$ and $\Theta_{G'_{m'}}(\gamma)\in R(G'_{m'})_{\lambda'}$ otherwise.
\end{theorem}

We are going to extend this theorem to the Witt tower $({\bf Sp}, {\bf O}^\epsilon_\rm{odd})$. Since $O^{+}_{2n'+1}$ and $O^{-}_{2n'+1}$ are isomorphic, we may simply write $\rm{Irr}(O_{2n'+1})$ instead of $\rm{Irr}(O^{+}_{2n'+1})$ and $\rm{Irr}(O^{-}_{2n'+1})$.
The main result of this section is the following theorem, which is compatible with Pan's decomposition formula (Theorem \ref{pan}).

\begin{theorem}\label{amr2}
Let $\lambda$ be an irreducible cuspidal representation of $Sp_{2n}$ and let $n'=n'^\epsilon(\lambda)$ be its first occurrence index in ${\bf O}^\epsilon_\rm{odd}$, so that $\lambda':=\Theta_{O^\epsilon_{2n'+1}}(\lambda)$ is an irreducible cuspidal representation of $O^\epsilon_{2n'+1}$. Then the following holds.

 (i) For  $\gamma \in \rm{Irr}(Sp_{2m})_{\lambda,\theta}$, $\Theta_{O^{\epsilon}_{2m'+1}}(\gamma)=0$ if $m'<n'$ and $\Theta_{O^{\epsilon}_{2m'+1}}(\gamma) \in R(O_{2m'+1})_{\lambda'}$ otherwise.

 (ii) For $\gamma' \in \rm{Irr}(O^\epsilon_{2m'+1})_{\lambda'}$, $\Theta_{Sp_{2m}}(\gamma')=0$ if $m<n$ and $\Theta_{Sp_{2m}}(\gamma') \in R(Sp_{2m})_{\lambda,\theta}$ otherwise.

(iii) For $\gamma \in \rm{Irr}(Sp_{2m})_{\lambda}$, $\Theta_{O^\epsilon_{2m'+1}}(\gamma)=0$ if $m'<n'$ and $\Theta_{O^{\epsilon}_{2m'+1}}(\gamma)
\in R(O_{2m'+1}^\epsilon)_{\lambda',\theta'}$ otherwise, where $\theta':=\theta'_{m-n,m'-n'}$ is defined by (\ref{tkl'}).

(iv)  For $\gamma' \in \rm{Irr}(O^\epsilon_{2m'+1})_{\lambda',\theta}$, $\Theta_{Sp_{2m}}(\gamma')=0$ if $m<n$ and $\Theta_{Sp_{2m}}(\gamma')\in R(Sp_{2m})_{\lambda, \theta'}$, where $\theta':=\theta_{m'-n',m-n}$ is defined by (\ref{tkl}).
\end{theorem}

\begin{proof}
We will only prove (i) and (iii). The proofs of (ii) and (iv)  are similar and will be left to the reader. It is parallel to Theorem \ref{amr} and we shall follow the calculations in \cite{AMR} and \cite{MVW} closely.

We first prove (i), by induction on $m$.

 $\bullet$ Suppose that $ m = n$ (i.e. $\gamma=\lambda$ is cuspidal) and $m'>n'$. It is known that (cf. \cite[Chap. 3]{MVW}) each constituent $\gamma'\in\Theta_{O^\epsilon_{2m'+1}}(\gamma)$  is noncuspidal. Let $j$ be the maximum integer such that $\gamma\in R^{O_{2m'+1}}_{O_{2(m'-j)+1}\times GL_{j}}(\lambda_{1}'\otimes \rho')$ with $\lambda_{1}'\in  \rm{Irr}(O_{2(m'-j)+1})$ cuspidal and $\rho'\in \rm{Irr}(GL_{j})$. Since $\gamma'\in\Theta_{O^\epsilon_{2m'+1}}(\gamma)$, one has
\[
\begin{aligned}
0&<\big(\omega^\epsilon_{m,m'},\gamma \otimes \gamma '\big)_{Sp_{2m}\times O_{2m'+1}} \\
&\leq \big(\omega^\epsilon_{m,m'},\gamma \otimes R^{O_{2m'+1}}_{O_{2(m'-j)+1}\times GL_{j}}(\lambda_{1}'\otimes \rho')\big)_{Sp_{2m}\times O_{2m'+1}}
\\&=\big({}^{*}R^{O_{2m'+1}}_{O_{2(m'-j)+1}\times GL_{j}}(\omega^\epsilon_{m,m'}),\gamma \otimes \lambda_{1}'\otimes \rho'\big)_{Sp_{2m}\times O_{2(m'-j)+1}\times GL_{j}}.
\end{aligned}
\]
Here ${}^*R$ standards for the Jacquet functor, which is adjoint to the induction functor $R$. We have the following decomposition (cf. \cite[Chap. 3, IV th.5]{MVW})
\[
\begin{aligned}
&{}^{*}R^{O_{2m'+1}}_{O_{2(m'-j)+1}\times GL_{j}}(\omega^\epsilon_{m,m'})
\\=&\bigoplus^{\min(m,j)}_{i=0}R^{Sp_{2m}\times GL_{j} \times O_{2(m'-j)+1}}_{Sp_{2(m-i)}\times GL_{i} \times ( GL_{j-i}\times GL_{i})\times O_{2(m'-j)+1}}(\omega^\epsilon_{m-i,m'-j} \otimes 1_{GL_{j-i}} \otimes \chi_{GL_{i}}R^{GL_{i}}).
\end{aligned}
\]
Hence $\big(\omega^\epsilon_{m,m'},\gamma \otimes \gamma '\big) $ is bounded by
\[
\sum^{\min(m,j)}_{i=0}\big(\omega^\epsilon_{m-i,m'-j} \otimes 1_{GL_{j-i}} \otimes \chi_{GL_{i}}R^{GL_{i}},{}^{*}R^{Sp_{2m}\times GL_{j} \times O_{2(m'-j)+1}}_{Sp_{2(m-i)}\times GL_{i} \times ( GL_{j-i}\times GL_{i})\times O_{2(m'-j)+1}}(\gamma \otimes \lambda_{1}'\otimes \rho')\big),
\]
where the scalar product in the $i$th summand is taken over the group $Sp_{2(m-i)}\times O_{2(m'-j)+1}\times GL_{j-i}\times GL_i\times GL_i$.
Since $\gamma=\lambda$ is cuspidal, the only nonzero term corresponds to $i = 0$, which implies that
\[
\big(\omega^\epsilon_{m,m'-j}\otimes 1_{GL_{j}},\gamma \otimes \lambda_{1}'\otimes \rho'\big) >0.
\]
It follows that $\rho'=1_{GL_{j}}$ and $\lambda_{1}'\in\Theta_{O^\epsilon_{2(m'-j)+1}}(\gamma)$. Because $\lambda_{1}'$ is cuspidal,
we must have $m'-j=n'$ and $\lambda_1'=\lambda'$, i.e.  $\gamma' \in \rm{Irr}(O_{2m'+1})_{\lambda'}$.

$\bullet$ Suppose that $ m > n$. If $\gamma\in \rm{Irr}(Sp_{2m})_{\lambda,\theta}$, then there exists $\gamma_{1} \in \rm{Irr}(Sp_{2(m-1)})_{\lambda,\theta}$ such that $\gamma\in R^{Sp_{2m}}_{Sp_{2(m-1)}\times GL_1}(\gamma_1\otimes \chi_{GL_1})$. For  $\gamma ' \in \Theta_{O^\epsilon_{2m'+1}}(\gamma)$ we have
\[
\begin{aligned}
0&<\big(\omega^\epsilon_{m,m'},\gamma \otimes \gamma '\big)_{Sp_{2m}\times O_{2m'+1}}
\\&\leq \big(\omega^\epsilon_{m,m'},R^{Sp_{2m}}_{Sp_{2(m-1)}\times GL_1}(\gamma_{1}\otimes\chi_{GL_1}) \otimes \gamma '\big)_{Sp_{2m} \times O_{2m'+1} }
\\&=\big({}^{*}R^{Sp_{2m}}_{ Sp_{2(m-1)}\times GL_1}(\omega^\epsilon_{m,m'}),\gamma_{1} \otimes\chi_{GL_1} \otimes\gamma '\big)_{ Sp_{2(m-1)} \times GL_1\times O_{2m'+1} }
\end{aligned}
\]
We have the decomposition
\[
\begin{aligned}
&{}^{*}R^{Sp_{2m}}_{ Sp_{2(m-1)}\times GL_1}(\omega^\epsilon_{m,m'})
\\=& \omega^\epsilon_{m-1,m'}\otimes \chi_{GL_1}+R^{ Sp_{2(m-1)}\times O_{2m'+1}\times GL_1}_{Sp_{2(m-1)}\times O_{2m'-1}\times GL_1\times GL_1}( \omega^\epsilon_{m-1,m'-1}\otimes \chi_{GL_1}R^{GL_1}).
\end{aligned}
\]
Hence $\big(\omega^\epsilon_{m,m'},\gamma \otimes \gamma '\big) $ is bounded by
\[
\begin{aligned}
& \big( \omega^\epsilon_{m-1,m'}\otimes \chi_{GL_1},\gamma_{1} \otimes \gamma '\otimes \chi_{GL_{1}}\big)_{Sp_{2(m-1)} \times O_{2m'+1}\times GL_1}
\\+&\big( \omega^\epsilon_{m-1,m'-1}\otimes \chi_{GL_{1}}R^{GL_{1}},\gamma_{1} \otimes {}^{*}R^{O_{2m'+1}}_{ O_{2m'-1}\times GL_1}(\gamma ')\otimes\chi_{GL_1}\big)_{Sp_{2(m-1)}\times O_{2m'-1}\times GL_1\times GL_1}.
\end{aligned}
\]
If the first term is nonzero, then $\gamma'\in \Theta_{O^\epsilon_{2m'+1}}(\gamma_{1})$.  By induction $\gamma'\in R(O_{2m'+1})_{\lambda'}$.
If the second term is nonzero, then there exists $\gamma_{1}' \in \rm{Irr}(O_{2m'-1})$ and $\rho' \in \rm{Irr}(GL_{1})$ such that $\gamma_1'\otimes\rho'\in{}^{*}R^{O_{2m'+1}}_{O_{2m'-1}\times GL_1}(\gamma ')$ and
\[
\big(\omega^\epsilon_{m-1,m'-1}\otimes \chi_{GL_{1}}R^{GL_{1}},\gamma_{1}\otimes\gamma_{1}'\otimes\rho'\otimes \chi_{GL_{1}}\big)_{Sp_{2(m-1)}\times O_{2m'-1}\times GL_1\times GL_1} \ne 0.
\]
 It follows that $\rho'=1$ and $\gamma_{1}'\in \Theta_{O^\epsilon_{2m'-1}}(\gamma_{1})$. By induction, $\gamma_1' \in \rm{Irr}(O_{2m'-1})_{\lambda'}$ and therefore $\gamma'\in \rm{Irr}(O_{2m'+1})_{\lambda'}$ by Frobenius reciprocity. This proves (i).

Next we prove (iii), again by induction on $m$. Assume that  $\gamma \in \rm{Irr}(Sp_{2m})_{\lambda}$ and $\gamma'\in\Theta_{O^\epsilon_{2m'+1}}(\gamma)$.
 The case $m=n$ is covered by (i), so we may assume that $m>n$. Then there exists $\gamma_{1} \in \rm{Irr}(Sp_{2(m-1)})_\lambda$  such that $\gamma\in R^{Sp_{2m}}_{ Sp_{2(m-1)}\times GL_1}(\gamma_{1}\otimes 1)$. Since $\gamma ' \in \Theta_{O^\epsilon_{2m'+1}}(\gamma)$, one has
\[
\begin{aligned}
0&<\big(\omega^\epsilon_{m,m'},\gamma \otimes \gamma '\big)_{Sp_{2m} \times O_{2m'+1} }
\\&\leqslant \big(\omega^\epsilon_{m,m'},R^{Sp_{2m}}_{Sp_{2(m-1)}\times GL_1}( \gamma_{1}\otimes 1) \otimes \gamma '\big)_{Sp_{2m} \times O_{2m'+1} }
\\&=\big({}^{*}R^{Sp_{2m}}_{Sp_{2(m-1)}\times GL_1}(\omega^\epsilon_{m,m'}), \gamma_{1} \otimes 1\otimes \gamma '\big)_{Sp_{2(m-1)}\times GL_1 \times O_{2m'+1} }.
\end{aligned}
\]
Similar to the proof of (i),
 $\big(\omega^\epsilon_{m,m'},\gamma \otimes \gamma '\big) $ is bounded by
\[
\begin{aligned}
& \big( \omega^\epsilon_{m-1,m'}\otimes \chi_{GL_1},\gamma_{1} \otimes \gamma '\otimes 1\big)_{Sp_{2(m-1)} \times O_{2m'+1}\times GL_1}
\\+&\big( \omega^\epsilon_{m-1,m'-1}\otimes \chi_{GL_{1}}R^{GL_{1}},\gamma_{1} \otimes {}^{*}R^{O_{2m'+1}}_{ O_{2m'-1}\times GL_1}(\gamma ')\otimes 1\big)_{Sp_{2(m-1)}\times O_{2m'-1}\times GL_1\times GL_1}.
\end{aligned}
\]
The first term  is zero. If the second term is nonzero, then there exists $\gamma_{1}' \in \rm{Irr}(O_{2m'-1})$ and $\rho' \in \rm{Irr}(GL_{1})$ such that $\gamma_{1}'\otimes\rho'\in {}^{*}R^{O_{2m'+1}}_{ O_{2m'-1}\times GL_1}(\gamma ')$ and
\[
\big( \omega^\epsilon_{m-1,m'-1}\otimes \chi_{GL_{1}}R^{GL_{1}},\gamma_{1}\otimes\gamma_{1}'\otimes\rho'\otimes 1\big)_{Sp_{2(m-1)}\times O_{2m'-1}\times GL_1\times GL_1} \ne 0.
\]
 It follows that $\rho'=\chi_{GL_{1}}$ and $\gamma_{1}'\in \Theta_{O^\epsilon_{2m'-1}}(\gamma_{1})$. By induction,
\[
\gamma_1' \in R^{O_{2m'-1}}_{O_{2n'+1}\times T_{m'-n'-1}}( \lambda'\otimes\theta'_{m-n-1,m'-n'-1}),
\]
hence by Frobenius reciprocity
\[
\gamma'\in R^{O_{2m'+1}}_{O_{2n'+1}\times T_{m'-n'}}(\lambda'\otimes\theta'_{m-n,m'-n'}).
\]
This finishes the proof of (iii).
\end{proof}


In particular, if we consider the Harish-Chandra series of cuspidal unipotent representations, then combining Theorem \ref{uni} we obtain the following immediate consequence.

\begin{corollary}\label{cor4.3}
Let $\lambda:=\lambda_k$ be the unique cuspidal unipotent representation of $Sp_{2n}$, $n=k(k+1)$, and $\lambda':=\lambda'^+_{k,\chi}$ be its theta lifting to $O^\epsilon_{2n+1}$. Then the following holds.

(i) For $\gamma \in \rm{Irr}(Sp_{2m})_{\lambda}$, $\Theta_{O^\epsilon_{2m'+1}}(\gamma)=0$ if $m'<m$ and
\[
\rm{Irr}\left(\Theta_{O^\epsilon_{2m'+1}}(\gamma)\right)\subset \rm{Irr}\left(R^{O_{2m'+1}}_{O_{2n+1}\times T_{m-n}\times T_{m'-m}}(\lambda'\otimes\theta_{m-n}\otimes 1)\right)
\]
otherwise.

(ii)  For $\gamma' \in \rm{Irr}(O^\epsilon_{2m'+1})_{\lambda',\theta}$, $\Theta_{Sp_{2m}}(\gamma')=0$ if $m<m'$ and
\[
\rm{Irr}\left(\Theta_{Sp_{2m}}(\gamma') \right)\subset \rm{Irr}\left(R^{Sp_{2m}}_{Sp_{2n}\times T_{m'-n}\times T_{m-m'}}(\lambda\otimes 1\otimes\theta_{m-m'})\right)
\]
otherwise.
\end{corollary}

We end this paper by some remarks on the open problem of finding the explicit theta correspondence of Harish-Chandra series, which can be reformulated as a correspondence of Weyl group representations. Explicit descriptions have been proved for unitary dual pairs
(\cite[Th\'eor\`eme 3.10]{AMR}) and  symplectic-even orthogonal dual pairs (\cite[Conjecture 3.11]{AMR}, proved recently in \cite{P3}). These results are also explicated recently in \cite{E} in order to pick up extremal components from the big theta lifting. We expect that a similar conjecture should hold for symplectic-odd orthogonal case as well according to the results in this section.

Let us give a few details. From Mackey theory  we know that irreducible constituents in the Harish-Chandra series of a cuspidal unipotent representation are parametrized by irreducible representations of a Weyl group. This observation leads to a connection with the Springer correspondence \cite{AKP, E}.
More precisely, let $\lambda_k$ be the unique cuspidal unipotent representation of a classical group $G_{n_k}$, where $n_k$ is explicated  in Theorem \ref{fun}. Namely one has the table
\[
\begin{tabular}{|l|l|l|}
\hline
$n_k$ &  Type of $G$ & Sign of $G$\\
\hline
$\lfloor k(k+1)/4 \rfloor$ & unitary & sgn$(-1)^{k(k+1)/2}$\\
$k(k+1)$ & symplectic & NA\\
$k(k+1)$ & odd orthogonal & $\pm$ \\
$k^2$ & even orthogonal & sgn$(-1)^k$ \\
\hline
\end{tabular}
\]
For $G_{n_k}\subset G_m$ in the same Witt tower,  the set $\rm{Irr}(G_m)_{\lambda_k}$ is parametrized by irreducible representations of
the group
\[
W_{G_m}(\lambda_{k}) = \left\{x \in N_{G_m}(L)/L \ | \ \lambda^{x}_{k} \cong \lambda_{k}\right\},
\]
where $L = G_{n_k}\times T_{m-n_k}$.
The uniqueness of $\lambda_{k}$ implies that the defining condition of $W_{G_m}(\lambda_k)$ is trivial, hence
\[
W_{G_m}(\lambda_{k})=N_{G_m}(L)/L\cong W_{m-n_k}
\]
 which is a Weyl group of type $B_{m-n_k}$.

By the work \cite{AM}, the theta correspondence takes $\lambda_k$ to $\lambda'_{k'}$, where $k'=k$ for symplectic-even orthogonal case and $k'=k\pm1$ for unitary case. Hence in these cases the theta correspondence between Harish-Chandra series of cuspidal unipotent representations can be characterized as a correspondence between representations of a pair of Weyl groups $(W_{m-n_k}, W_{m'-n_{k'}})$.
As we mentioned above, explicit correspondence of representations of such a pair $(W_l, W_{l'})$ has been proved for unitary dual pairs
and  symplectic-even orthogonal dual pairs in \cite{AMR} and  \cite{P3}, respectively.
The set $\rm{Irr}(W_{l})$ has a well-known parametrization by bipartitions of $l$ (see e.g. \cite[Theorem 5.5.6]{GP}).  In \cite{E} J. Epequin Chavez further explicates the above correspondence using bipartitions. Among other applications, an observation we make from the results in \cite{E} is that in contrast to the $p$-adic case \cite{SZ}, the conservation relations do not hold in general for noncuspidal representations of finite dual pairs. It should be very interesting to understand the conservation type relations for finite dual pairs beyond the cuspidal case.

For the symplectic-odd orthogonal case, recall from Theorem \ref{thm3.7} that  $\lambda_k$ is the unique cuspidal unipotent representation of $Sp_{2n_k}$, and $\lambda'^+_{k,\chi}$ is the unique cuspidal $\theta$-representation of $O_{2n_k+1}$ with trivial central character, $n_k=k(k+1)$.
In a similar manner as in the unipotent case, the Harish-Chandra series $\rm{Irr}(Sp_{2m})_{\lambda_k,\theta}$ and $\rm{Irr}(O_{2m'+1})_{\lambda'^+_{k,\chi}}$ are in bijection with $\rm{Irr}(W_{m-n_k})$ and $\rm{Irr}(W_{m'-n_k})$ respectively. By Theorem \ref{amr2} (i) (ii), we expect that a conjecture analogous to \cite[Conjecture 3.11]{AMR} should hold for the theta correspondence between $\rm{Irr}(Sp_{2m})_{\lambda_k,\theta}$ and $\rm{Irr}(O_{2m'+1})_{\lambda'^+_{k,\chi}}$.

\end{document}